\documentclass[11pt,twoside,reqno]{amsart}

\usepackage{microtype}
\usepackage{cite}
\usepackage[OT1]{fontenc}
\usepackage{type1cm}
\usepackage{amssymb}
\usepackage{dsfont}
\usepackage{enumitem}
\usepackage{comment}
\usepackage{xcolor}
\usepackage{lmodern}

\usepackage{geometry}
\usepackage{hyperref}
\hypersetup{
  colorlinks=true,
  linkcolor=black,
  anchorcolor=black,
  citecolor=black,
  filecolor=black,      
  menucolor=red,
  runcolor=black,
  urlcolor=black,
}

\numberwithin{equation}{section}

\theoremstyle{plain}
\newtheorem{theorem}{Theorem}[section]

\newtheorem{lemma}[theorem]{Lemma}

\theoremstyle{remark}

\theoremstyle{definition}

\newcommand{\HH}{\mathcal{H}}

\newcommand{\R}{\mathbb{R}}

\newcommand{\Z}{\mathbb{Z}}
\newcommand{\N}{\mathbb{N}}

\newcommand{\roo}{\varrho}

\newcommand{\dd}{\,\mathrm{d}}

\renewcommand{\ge}{\geqslant}
\renewcommand{\le}{\leqslant}
\renewcommand{\geq}{\geqslant}
\renewcommand{\leq}{\leqslant}

\DeclareMathOperator{\dimm}{dim_M}

\DeclareMathOperator{\dimh}{dim_H}

\DeclareMathOperator{\dist}{dist}

\DeclareMathOperator{\diag}{diag}

\DeclareMathOperator{\spt}{spt}

\DeclareMathOperator{\id}{Id}
\renewcommand{\atop}[2]{\genfrac{}{}{0pt}{}{#1}{#2}}

\begin{document}

\title{Covering number on inhomogeneous graph-directed self-similar sets}

\author{Bal\'azs B\'ar\'any}
\address[Bal\'azs B\'ar\'any]
        {Department of Stochastics \\
        	Institute of Mathematics \\
        	Budapest University of Technology and Economics \\
        	M\H{u}egyetem rkp. 3 \\
        	H-1111 Budapest,
        	Hungary}
\email{balubsheep@gmail.com}

\author{Antti K\"aenm\"aki}
\address[Antti K\"aenm\"aki]
        {Alfr\'ed R\'enyi Institute of Mathematics \\
         Hungarian Academy of Sciences \\ 
         Budapest \\ 
         Hungary}
\email{kaenmaki@renyi.hu}

\author{Petteri Nissinen}
\address[Petteri Nissinen]
        {Department of Physics and Mathematics \\
         University of Eastern Finland \\
         P.O.\ Box 111 \\
         FI-80101 Joensuu \\
         Finland}
\email{pettern@student.uef.fi}

\date{\today}
\thanks{B. B\'ar\'any was supported by the grants NKFI FK134251, K142169, and the grant NKFI KKP144059 ``Fractal geometry and applications''.}

\begin{abstract}
  For a strongly connected inhomogeneous graph-directed self-similar set $K^C$ satisfying the strong open set condition, we characterize the asymptotic behaviour of the $r$-covering number $N_r(K^C)$ as $r \downarrow 0$ in terms of the Minkowski dimension $s_0(G)$ of the attractor. If $\int_0^\infty e^{-s_0(G)t}N_{e^{-t}}(C_i)\dd t<\infty$ for all vertices $i$, then $e^{-s_0(G)t}N_{e^{-t}}(K^C)$ has a limit as $t\to\infty$, which is a positive constant when the log-contraction group $G_M$ is $\R$ and a positive periodic function when $G_M$ is a lattice; if the integral diverges for some $i$, the limit is infinite. 
\end{abstract}

\maketitle

\section{Introduction}

Let $\Phi = (f_1,\ldots,f_\kappa)$ be a tuple of contractive similarities acting on $\R^d$ such that $f_i(x)=\roo_iO_ix+t_i$, where $0<\roo_i<1$ is a contraction ratio, $O_i$ is an orthogonal $d \times d$-matrix, and $t_i \in \R^d$ is a translation vector for all $i \in \{1,\ldots,\kappa\}$. A classical result of Hutchinson \cite{Hutchinson1981} shows that for each such $\Phi$ there exists a unique non-empty compact invariant set $K \subset \R^d$, called the \emph{self-similar set}, such that
\begin{equation} \label{eq:self-similar-set-def}
  K = \bigcup_{i=1}^\kappa f_i(K). 
\end{equation}
For a bounded set $A\subset\R^d$, let $N_r(A)$ be the \emph{$r$-covering number} of $A$, i.e.
\begin{equation*}
  N_r(A) = \min\biggl\{ k \in \N : A \subset \bigcup_{i=1}^k B(x_i,r) \text{ for some } x_1,\ldots,x_k \in \R^d \biggr\},
\end{equation*}
is the least number of closed balls of radius $r>0$ needed to cover $A$. Recall that, by Falconer \cite[Theorem 4]{Falconer1989}, the \emph{Minkowski dimension} of a self-similar $K$,
\begin{equation} \label{eq:minkowski-self-similar}
  \dimm(K) = \lim_{r \downarrow 0} \frac{\log N_r(K)}{\log r^{-1}},
\end{equation}
exists. A self-similar set $K$ satisfies the \emph{strong separation condition}, if $f_i(K) \cap f_j(K) = \emptyset$ for all $i,j \in \{1,\ldots,\kappa\}$ with $i \ne j$. Under the strong separation condition, Lalley \cite[Theorem 1]{Lalley1988} managed to give more precise information in \eqref{eq:minkowski-self-similar} by studying the convergence of
\begin{equation} \label{eq:covering-number-convergence}
  \frac{N_r(K)}{r^{-\dimm(K)}}
\end{equation}
as $r \downarrow 0$. In this paper, our goal is to generalize Lalley's result for more general systems under a more relaxed separation condition. Existing results either assume strong separation \cite{Lalley1988} or finite ramification \cite{HamblyNyberg2003}, and do not cover inhomogeneous graph-directed systems; the convergence of \eqref{eq:covering-number-convergence} in that setting has remained open. Our main results, Theorems \ref{thm:GIFS-renewal-result} and \ref{thm:GIFS-renewal-result2} below, completely characterize the convergence of \eqref{eq:covering-number-convergence} for strongly connected inhomogeneous graph-directed self-similar sets satisfying the strong open set condition.

Graph-directed self-similar sets generalize self-similar sets. They are defined by contractive similarities determined by a directed multigraph between non-empty compacts sets in $\R^d$. Such a configuration is called a Mauldin-Williams graph. More precisely, let $(V,E)$ be a directed multigraph with a set $V$ of vertices and with a multiset $E$ of directed edges. For every vertex $i\in V$, there exists a bounded set $X_i\subset\R^d$ such that $\overline{X_i^o}=X_i$, and for every edge $e\in E$, let $S_{e}\colon X_{t(e)}\to X_{i(e)}$ be a contractive similarity associated to the edge $e$ from the initial vertex $i(e)$ to the terminal vertex $t(e)$. The Mauldin-Williams graph is then $G=((V, E), (X_i)_{i\in V}, (S_e)_{e\in E})$. By Mauldin and Williams \cite[Theorem~1]{MR961615}, there exists a unique list $(K_i)_{i\in V}$ of non-empty compact subsets of $\R^d$, called the \emph{graph-directed self-similar sets}, such that
\begin{equation} \label{eq:gd-self-similar-set-def}
  K_i = \bigcup_{e\in E\,:\,i(e)=i} S_{e}(K_{t(e)})
\end{equation}
for all $i\in V$. The precise definition will be given in \S \ref{sec:GDIFS}. Note that in the case of one vertex and $\kappa$ many edges, the graph-directed self-similar set satisfies \eqref{eq:self-similar-set-def}. Write $K = \bigcup_{i\in V} K_i$. Hambly and Nyberg \cite{HamblyNyberg2003} studied the asymptotic behaviour of \eqref{eq:covering-number-convergence} for graph-directed self-similar sets under the strong open set condition and the so-called finite ramification.

Inhomogeneous graph-directed self-similar sets $(K_i^C)_{i\in V}$ are defined as graph-directed self-similar sets but with a list $C=(C_i)_{i\in V}$ of compact condensation sets. By Dubey and Verma \cite[Theorem~3.6]{dubey2023fractal}, there exists a unique list $(K_i^C)_{i \in V}$ of non-empty compact subsets of $\R^d$, called the \emph{inhomogeneous graph-directed self-similar sets}, such that
\begin{equation*}
	K_i^C = \bigcup_{e\in E\,:\,i(e)=i} S_{e}(K_{t(e)}^C)\cup C_{i}
\end{equation*}
for all $i \in V$. See \S \ref{sec:inhomog-gd} for a precise definition. Write also $K^C = \bigcup_{i\in V} K_i^C$. Note that if $C_i=\emptyset$ for all $i\in V$, then the inhomogeneous graph-directed self-similar set satisfies \eqref{eq:gd-self-similar-set-def}. Dubey and Verma \cite{dubey2023fractal} studied the Minkowski dimension of inhomogeneous graph-directed self-similar sets satisfying the strong open set condition.

Let us next state our main results. For the definitions of the strongly connected Mauldin-Williams graph, the strong open set condition (SOSC), and the strong condensation open set condition (SCOSC), the reader is referred to \S \ref{sec:GDIFS} and \S \ref{sec:inhomog-gd}. Furthermore, let $s_0(G)$ be the Minkowski dimension of the corresponding graph-directed self-similar sets defined in \eqref{eq:s0-def} and let $G_M$ be the smallest closed group generated by the logarithms of the contracting ratios defined in \eqref{eq:GMgraphdir}. The first main result covers the case where the condensation sets are small.

\begin{theorem} \label{thm:GIFS-renewal-result}
  Suppose that $(G,C)$ is a strongly connected inhomogeneous Mauldin-Williams graph satisfying the SOSC such that
  \begin{equation} \label{eq:condensation-finite}
    \int_0^\infty e^{-s_0(G)t}N_{e^{-t}}(C_i)\dd t<\infty
  \end{equation}
  for all $i \in V$. Then precisely one of the following two statements hold:
  \begin{enumerate}
    \item $G_M={\R}$ and for every $i\in V$ there exists a constant $h_i > 0$ such that
      \begin{equation*}
        \lim_{t\to \infty}\frac{N_{e^{-t}}(K_i^C)}{h_ie^{s_0(G)t}}=1.
      \end{equation*}
      In particular, $\lim\limits_{t\to \infty}e^{-s_0(G)t}N_{e^{-t}}(K^C) = \sum\limits_{i\in V} h_i$.
    \item $G_M=\langle \{\tau\}\rangle$ and for every $i\in V$ there exist $\delta>0$ and a $\tau$-periodic function $h_i \colon {\R}\to [\delta, \infty)$ such that
    \begin{equation*}
      \lim_{n\to \infty}\frac{N_{e^{-(n\tau+y)}}(K_i^C)}{h_i(y)e^{s_0(G)(n\tau+y)}}=1
    \end{equation*}
    for all $y\in [0, \tau)$. In particular, $\lim\limits_{n\to \infty}e^{-s_0(G)(n\tau+y)}N_{e^{-(n\tau+y)}}(K^C) = \sum\limits_{i\in V} h_i(y)$.
  \end{enumerate}
\end{theorem}

The second main result deals with the case where there are large condensation sets.

\begin{theorem} \label{thm:GIFS-renewal-result2}
  Suppose that $(G,C)$ is a strongly connected inhomogeneous Mauldin-Williams graph satisfying the SCOSC such that
  \begin{equation} \label{eq:condensation-infty}
    \int_0^\infty e^{-s_0(G)t}N_{e^{-t}}(C_i)\dd t=\infty
  \end{equation}
  for some $i \in V$. Then
  \begin{equation*}
    \lim_{t\to \infty}\frac{N_{e^{-t}}(K_i^C)}{e^{s_0(G)t}}=\infty
  \end{equation*}
  for all $i\in V$.
\end{theorem}

To summarise, our main results show that the condensation set either does not affect the asymptotic behaviour of the covering number or forces it to blow up. Furthermore, since our results are formulated for inhomogeneous graphs, they can also be applied to homogeneous graph-directed self-similar sets that are not strongly connected but satisfy the SOSC, in the following sense: the attractor in each connected component can be expressed as a strongly connected inhomogeneous self-similar set in which the condensation set either satisfies assumption \eqref{eq:condensation-infty} or $\dimm(C_i) < s_0(G)$ for all $i \in V$, and in particular assumption \eqref{eq:condensation-finite} holds. Regarding graph-directed self-similar sets satisfying the strong open set condition, Hambly and Nyberg \cite{HamblyNyberg2003} obtained a more refined description with precise growth rates which depend on height of the basic classes, under the finite ramification assumption that the overlap of two distinct first-level cylinders is finite. In the strongly connected case, we are able to remove this finite ramification hypothesis. In the non-strongly connected inhomogeneous case, however, the interaction between different components and the effect of the condensation set can produce additional growth phenomena, and it is not clear that one can obtain similarly sharp asymptotics as in \cite{HamblyNyberg2003} by our methods. A full treatment of this reducible case would require a separate analysis and lies beyond the scope of the present paper.

The remainder of the article is organized as follows: In \S \ref{sec:preli}, we introduce the graph-directed iterated function systems, define an inhomogeneous version of it, and recall the vector-valued renewal theorem. In \S \ref{sec:SC-inhomog}, we prove Theorems \ref{thm:GIFS-renewal-result} and \ref{thm:GIFS-renewal-result2}.

\section{Preliminaries} \label{sec:preli}

In this section, we recall the definition of a graph-directed iterated function system and present the vector-valued renewal theorem of Lau, Wang, and Chu \cite{MR1367690}.

\subsection{Graph-directed iterated function systems} \label{sec:GDIFS}
Let $(V, E)$ be a directed multigraph, where $V$ is the set of vertices and $E$ is the multiset of edges. For an edge $e\in E$, let us denote its initial vertex by $i(e)$ and by $t(e)$ its terminal vertex. If $i, t \in V$ are vertices then we denote the set of edges with initial vertex $i$ by $E_i=\{e\in E:i(e)=i\}$, the set of edges from $i$ to $t$ by $E_{it}=\{e\in E_i:t(e)=t\}$. A list $G=((V, E), (X_i)_{i\in V}, (S_e)_{e\in E})$ with the following three conditions:
\begin{enumerate}[label=(G\arabic*), ref=(G\arabic*)]
	\item\label{it:G1} $(V, E)$ is a directed multigraph,
	\item\label{it:G2} $(X_i)_{i\in V} \in (\R^d)^N$ is a list of non-empty compact subsets of $\R^d$ with $\overline{X_i^o}=X_i$,
	\item\label{it:G3} $S_e \colon X_{t(e)} \to X_{i(e)}$ is a contractive similarity with contraction ratio $0<r_e<1$ for all $e\in E$,
\end{enumerate}
is called a \emph{Mauldin-Williams graph}.

If $G=((V, E), (X_i)_{i\in V}, (S_e)_{e\in E})$ is a Mauldin-Williams graph, then, by \cite[Theorem 1]{MR961615}, there exists a unique list $(K_i)_{i\in V}$ of non-empty compact invariant subsets of $X_i$ satisfying
\begin{equation} \label{invariantlist}
  K_i=\bigcup_{e\in E_i} S_e (K_{t(e)})
\end{equation}
for all $i\in V$. The sets in $(K_i)_{i\in V}$ are called \emph{graph-directed self-similar sets} associated with $G$. It is a common practice to embed the sets $X_i$ into a single $\R^d$ such that $X_i^o\cap X_j^o = \emptyset$ whenever $i\neq j$. This allows us to define $K=\bigcup_{i\in V} K_i$, where the union is ``essentially'' pairwise disjoint. We say that $G$ satisfies the \emph{open set condition (OSC)} if there exist a list $(U_i)_{i\in V}$ of sets such that for all $i\in V$ the following three assumptions holds: each $U_i$ is non-empty open bounded subsets of $X_i$,
\begin{equation} \label{OSC}
  \bigcup_{e\in E_i} S_e (U_{t(e)})\subset U_i,
\end{equation}
and
\begin{equation*}
  S_e(U_{t(e)})\cap S_{e'}(U_{t(e')})=\emptyset
\end{equation*}
for all $e, e'\in E_i$ with $e\neq e'$. Furthermore, if $U_i \cap K_i \neq \emptyset$ for all $i\in V$, then we say that $G$ satisfies the \emph{strong open set condition (SOSC)}. It follows from \eqref{invariantlist} and \eqref{OSC} that $K_i\subset \overline{U_i}$ for all $i\in V$. Clearly, if a Mauldin-Williams graph satisfies the SOSC with open sets $(U_i)_{i \in V}$, then, without loss of generality, we may assume that $X_i=\overline{U_i}$.

A list $\gamma=(e_1, \ldots, e_n)$ of consecutive edges, which satisfies $t(e_1)=i(e_2)$, $t(e_2)=i(e_3)$, \ldots, $t(e_{n-1})=i(e_n)$, is called a \emph{path}. For a path $\gamma=(e_1, \ldots, e_n)$, we define $S_\gamma = S_{e_1} \circ \cdots \circ S_{e_n}$ and $r_\gamma = r_{e_1}\cdots r_{e_n}$. The number of edges in a path is its \emph{length}. A path $\gamma=(e_1, \ldots, e_n)$ is called a \emph{cycle} if $t(e_n)=i(e_1)$, i.e.\ the terminal vertex of $\gamma$, denoted by $t(\gamma)$, is equal to the initial vertex of $\gamma$, denoted by $i(\gamma)$. A cycle $\gamma=(e_1, \ldots, e_n)$ is \emph{simple} if all the initial vertices $i(e_1), \ldots, i(e_n)$ are distinct. Let us denote the set of paths of length $n$ beginning at $i$ and terminating at $j$ by
$$
  \Gamma_{i,j}^n=\{\gamma=(e_1,\ldots,e_n)\in E^n:\gamma\text{ is a path such that }i(e_1)=i \text{ and } t(e_n)=j\}
$$
and write $\Gamma_{i,j}^* = \bigcup_{n \in \N} \Gamma_{i,j}^n$. We use the convention that $\varnothing$ is an element of $\Gamma_{i,j}^*$ and has the property that $S_\varnothing$ is the identity map $\id|_{X_i}$. Similarly, let $\Gamma_i^n=\bigcup_{j\in V}\Gamma_{i,j}^n$ be the set of $n$-length paths beginning at $i$ and write $\Gamma_{i}^* = \bigcup_{n \in \N} \Gamma_{i}^n$. If $\gamma = (e_1,\ldots,e_n) \in \Gamma_i^n$, then we write $\gamma^- = (e_1,\ldots,e_{n-1})$. Let $\Gamma^n = \bigcup_{i\in V} \Gamma_i^n$ and $\Gamma^* = \bigcup_{n \in \N} \Gamma^n$. Write $|\gamma|$ for the length of $\gamma \in \Gamma^*$. We say that $G$ is \emph{strongly connected} if for each pair of vertices $i$ and $j$, there is path from $i$ to $j$.  In particular, if $G$ is strongly connected then there exists $r>0$ such that for every $i,j\in V$, $\Gamma_{i,j}^r\neq\emptyset$. For any two paths $\gamma=(e_1,\ldots,e_n),\gamma'=(e_1',\ldots,e_n')\in\Gamma^*$ let $\gamma\wedge\gamma'=(e_1,\ldots,e_{|\gamma\wedge\gamma'|})$, where $|\gamma\wedge\gamma'|=\min\{k\geq0:e_{k+1}\neq e_{k+1}'\}$, the common part of the paths $\gamma,\gamma'$. We use the convention that if $|\gamma\wedge\gamma'|=0$ then $\gamma\wedge\gamma'=\varnothing$.

We also define the set of infinite length paths beginning at $i$ by
$$
  \Gamma_i=\{\gamma=(e_1,e_2,\ldots)\in E^\N:\gamma\text{ satisfies }i(e_1)=i \text{ and } t(e_n)=i(e_{n+1})\text{ for all }n\in\N\}.
$$
If $\gamma = (e_1,e_2,\ldots) \in E^\N$, then we write $\gamma|_n = (e_1,\ldots,e_n) \in E^n$ for all $n \in \N$. The \emph{canonical projection} $\Pi_i\colon\Gamma_i\to K_i$ is defined by the relation
$$
  \{\Pi_i(\gamma)\}=\bigcap_{n=1}^\infty S_{\gamma|_n}(K_{t(\gamma|_n)})
$$
for all $\gamma \in E^\N$. Note that each $\Pi_i$ is surjective. Defining $\Gamma=\bigcup_{i\in V}\Gamma_i$ and $\Pi\colon\Gamma\to K$ by $\Pi(\gamma)=\Pi_{i(\gamma)}(\gamma)$, we have $\Pi(\Gamma)=K$.

One can identify the finite set of vertices with positive integers, i.e. we may assume that $V=\{1,\ldots,N\}$. A non-negative $N \times N$ matrix $A$ is \emph{irreducible} if for all $i, j \in\{1, \ldots, N\}$ there exist $k\in {\N}$ such that $(A^k)_{ij}>0$. Here $A_{ij}$ denotes the $(i,j)$ element of a matrix $A$. For a given Mauldin-Williams graph $G$, define $A_{G}^{s}=( \sum_{e\in E_{ij}} r_e^s)_{i, j\in\{1, \ldots, N\}}$ for all $s \ge 0$. It is easy to see that $A_G^s$ is irreducible for all $s \ge 0$ if and only if $G$ is strongly connected. Recall that, by the Perron-Frobenius theorem, the \emph{spectral radius} $\rho(A)$ of an $N \times N$ matrix $A$ is the largest eigenvalue of $A$ in modulus. Let $s_0(G)\geq 0$ be the unique solution of
\begin{equation} \label{eq:s0-def}
  \rho(A_{G}^s)=1.
\end{equation}
The following theorem follows from \cite[Theorem 3]{MR961615} and \cite[Corollary 3.5]{MR1449135}.

\begin{theorem} \label{thm:MW-dimension}
  Let $G$ be a strongly connected Mauldin-Williams graph and $(K_i)_{i\in V}$ the associated list of graph-directed self-similar sets. If $G$ satisfies the OSC, then
  \begin{equation*}
  \dimm(K) = \dimh(K) = \dimh(K_i) = \dimm(K_i) = s_0(G)
  \end{equation*}
  and $0<\HH^{s_0(G)}(K_i)<\infty$ for all $i\in V$.
\end{theorem}

We also recall the following theorem which is a combination of \cite[Theorem 2.2.6]{Wang1994} and \cite[Theorem 5]{MR961615}.

\begin{theorem} \label{thm:MW-OSC-SOSC}
  Let $G$ be a strongly connected Mauldin-Williams graph and $(K_i)_{i\in V}$ the associated list of graph-directed self-similar sets. Then
  \begin{equation*}
    OSC \quad\Leftrightarrow\quad SOSC \quad\Leftrightarrow\quad 0<\HH^{s_0(G)}(K_i)<\infty
  \end{equation*}
  for all $i \in V$.
\end{theorem}

\subsection{Inhomogeneous graph-directed systems} \label{sec:inhomog-gd}
We introduce an extension of the Mauldin-Williams graph which is also our main interest. Let $G$ be a Mauldin-Williams graph defined in \S \ref{sec:GDIFS}. If there exists a list $C=(C_i)_{i \in V}$ of compact sets such that $C_i \subset X_i$ for all $i \in V$, then we see that there exists a unique list $(K_i^C)_{i\in V}$ of non-empty compact invariant subsets of $X_i$ satisfying
\begin{equation*} \label{invariantlist2}
  K_i^C=\bigcup_{e\in E_i} S_e (K_{t(e)}^C) \cup C_i
\end{equation*}
for all $i\in V$. Note that if $C_i = \emptyset$ for all $i \in V$ and we denote $(\emptyset)_{i \in V}$ by $\emptyset$, then $K_i^\emptyset$ is the set satisfying \eqref{invariantlist} and will be denoted by $K_i$. Recall that we introduced the convention that $\varnothing\in\Gamma_i^*$ and $S_\varnothing=\id|_{X_i}$, so it is straightforward to see by the definition that
\begin{equation} \label{eq:inhomogeneous-def2}
  K_i^C=K_i\cup \bigcup_{\gamma\in \Gamma_i^*} S_{\gamma}(C_{t(\gamma)}).
\end{equation}
We say that $(G, C)$ is an \emph{inhomogeneous Mauldin-Williams graph}. Calling $(G, C)$ \emph{strongly connected} obviously then means that $G$ is strongly connected. The sets in $(K_i^C)_{i\in V}$ are called \emph{inhomogenous graph directed self-similar sets} associated with $(G, C)$ and, again embedding the sets $X_i$ into a single $\R^d$ such that $X_i^o \cap X_j^o = \emptyset$ whenever $i\neq j$, their union is denoted by $K^C=\bigcup_{i\in V} K_i^C$.

We say that $(G, C)$ satisfies the \emph{strong open set condition (SOSC)} if $G$ satisfies the SOSC with open sets $(U_i)_{i \in V}$ and $C_i \subset \overline{U_i}$ for all $i \in V$. Observe that if $(G, C)$ satisfies the SOSC, then $K_i^C \subset \overline{U_i}$ and we may assume that $X_i = \overline{U_i}$ for all $i\in V$. Furthermore, following \cite{KaenmakiLehrback2017}, we say that $(G,C)$ satisfies the \emph{strong condensation open set condition (SCOSC)} if $G$ satisfies the SOSC with open sets $(U_i)_{i \in \{1,\ldots,N\}}$ and $C_i \subset \overline{U_i} \setminus \bigcup_{e \in E_i} S_e(U_{t(e)})$ for all $i \in \{1,\ldots,N\}$.

\subsection{Vector-valued renewal theorem} \label{sec:vector-renewal}
We say that a locally finite Borel regular measure is a \emph{Radon measure}. For Radon measures $\mu_1$ and $\mu_2$ on ${\R}$, we define the \emph{convolution} of $\mu_1$ and $\mu_2$ by
\begin{equation*}
(\mu_1*\mu_2)(A)=\iint \mathds{1}_A(x+y)\dd\mu_1(x)\dd\mu_2(y)
\end{equation*}
for all Borel sets $A \subset \R$. Observe that the convolution is a Radon measure on ${\R}$, and that we furthermore have
\begin{align*}
  \spt(\mu_1+\mu_2) &= \spt(\mu_1)\cup\spt(\mu_2),\\
  \spt(\mu_1*\mu_2) &= \spt(\mu_1)+\spt(\mu_2), 
\end{align*}
where $(\mu_1+\mu_2)(A) = \mu_1(A)+\mu_2(A)$ and $A+B=\{x+y : x\in A \text{ and } y\in B\}$ for all $A,B \subset \R$. Moreover, if $\spt(\mu_1), \spt(\mu_2) \subset [0, \infty)$, then clearly $\spt(\mu_1+\mu_2) = \spt(\mu_1)\cup\spt(\mu_2) \subset [0, \infty)$ and $\spt(\mu_1*\mu_2) = \spt(\mu_1)+\spt(\mu_2)\subset [0, \infty)$.

Let us define a \emph{matrix valued Radon measure} by setting
\begin{equation} \label{matrix valued measure}
  M =
  \begin{pmatrix}
    \mu_{11} & \cdots & \mu_{1N} \\
    \vdots & \ddots &\vdots \\
    \mu_{N1} & \cdots &\mu_{NN}
  \end{pmatrix},
\end{equation}
where each $\mu_{ij}$, $i,j \in \{1,\ldots, N\}$, is a Radon measure on ${\R}$ such that $\spt(\mu_{ij}) \subset [0, \infty)$.
If $\gamma=((i_1,i_2),(i_2,i_3), \ldots, (i_{k-1},i_k))$ is a path, then we write 
\begin{equation*}
  \mu_{\gamma}=\mu_{i_1i_2}*\mu_{i_2i_3}*\cdots*\mu_{i_{k-1}i_k}.
\end{equation*}
Note that $\spt(\mu_{i_1i_2}*\mu_{i_2i_3}*\cdots*\mu_{i_{k-1}i_k})\subset [0, \infty)$. If $M = (\mu_{ij})_{i,j \in \{1,\ldots,N\}}$ and $P = (\nu_{ij})_{i,j \in \{1,\ldots,N\}}$ are matrix valued Radon measures, then we define their convolution by
\begin{equation*}
  M*P =
  \begin{pmatrix}
    \sum_{l=1}^N \mu_{1l}*\nu_{l1} & \cdots & \sum_{l=1}^N \mu_{1l}*\nu_{lN} \\ \vdots & \ddots &\vdots \\  \sum_{l=1}^N \mu_{Nl}*\nu_{l1} & \cdots &\sum_{l=1}^N \mu_{Nl}*\nu_{lN}
  \end{pmatrix}.
\end{equation*}
Let $M^{*0} = \diag(\delta_0,\ldots,\delta_0)$ and define recursively $M^{*k} = M^{*(k-1)} * M$ for all $k \in \N$. Note that the measures $M^{*k}_{ij}$ are Radon on ${\R}$ such that $\spt(M^{*k}_{ij}) \subset [0, \infty)$. We also write
\begin{equation*}
  U=\sum_{k=0}^{\infty} M^{*k}=
  \begin{pmatrix}
    \sum_{k=0}^{\infty} M_{11}^{*k} & \cdots & \sum_{k=0}^{\infty} M_{1N}^{*k} \\ \vdots & \ddots  &\vdots \\  \sum_{k=0}^{\infty} M_{N1}^{*k} & \cdots & \sum_{k=0}^{\infty} M_{NN}^{*k}
  \end{pmatrix}
\end{equation*}
and
\begin{align}
\label{support set}
  G_M=\Bigl\langle\bigcup\{ \spt(\mu_{\gamma}): {\gamma} \text{ is a simple cycle}\}\Bigr\rangle.
\end{align}
Assuming that the set $G_M$ is non-empty and nontrivial (i.e.\ $G_M \neq \{0\}$), then there are two possibilities: either $G_M={\R}$ or $G_M=\langle \{\tau\} \rangle$ for some $\tau>0$. Indeed, if a non-empty $B$ is a nontrivial subgroup of the additive group of real numbers, then one of the following holds: $B$ is dense in ${\R}$ or $B=\tau\Z=\{\tau k : k\in {\Z}\}$ for some $\tau>0$. Thus closed non-empty nontrivial subgroups of the additive group of real numbers are ${\R}$ and $\tau{\Z}$.

Recall that
\begin{equation*}
  (f*\mu)(x)=\int f(x-y)\dd\mu(y)
\end{equation*}
for all Radon measures $\mu$ on $\R$ and $f \colon \R \to \R$. If $M = (\mu_{ij})_{i,j \in \{1,\ldots,N\}}$ is a matrix valued Radon measure and $f=(f_1,\ldots,f_N) \colon \R \to \R^N$, then we define
\begin{align*}
(f*M)(x) &= (f_1,\ldots, f_N)*
\begin{pmatrix}
    \mu_{11} & \cdots & \mu_{1N} \\
    \vdots &  \ddots &\vdots \\
    \mu_{N1} & \cdots &\mu_{NN}
  \end{pmatrix}(x)\\
  &= \biggl( \sum_{l=1}^N f_l*\mu_{l1}(x),\ldots,\sum_{l=1}^N f_l*\mu_{lN}(x) \biggr)
\end{align*}
Let $L = (L_1,\ldots,L_N) \colon {\R} \to {\R}^N$ be a function vanishing for $x<0$ and let us consider the inhomogeneous convolution equation of the form
\begin{equation} \label{reneq}
  f(x) = (f*M)(x) + L(x),
\end{equation}
where $M = (\mu_{ij})_{i,j \in \{1,\ldots,N\}}$ is a matrix valued Radon measure and $f \colon {\R} \to {\R}^N$. Furthermore, we assume that the component functions $L_1, \ldots, L_N \colon {\R} \to {\R}$ of $L$ are \emph{directly Riemann integrable}, that is, each $L_i$ is Riemann integrable on finite closed intervals and 
\begin{equation}\label{eq:dirRiemann}
  \sum_{k=0}^{\infty} \sup_{t\in [k, k+1]} |L_i(t)|<\infty.
\end{equation}
We say that $L$ is directly Riemann integrable if all of its component functions are. Finally, we denote the cumulative distribution of $M$ by
\begin{equation*}
  F_M(x) =
  \begin{pmatrix}
    F_{11}(x) & \cdots & F_{1N}(x) \\
    \vdots & \ddots  &\vdots \\
    F_{N1}(x) & \cdots & F_{NN}(x)
  \end{pmatrix}
	=
	\begin{pmatrix}
    \mu_{11}([0, x]) & \cdots & \mu_{1N}([0, x]) \\
    \vdots & \ddots  &\vdots \\
    \mu_{N1}([0, x]) & \cdots & \mu_{NN}([0, x])
  \end{pmatrix}
\end{equation*}
for all $x \ge 0$ and the matrix $E$ of first moments is
\begin{equation*}
  E =
  \begin{pmatrix}
    m_{11} & \cdots & m_{1N} \\ 
    \vdots & \ddots & \vdots \\ 
    m_{N1} & \cdots & m_{NN}
  \end{pmatrix}
  =
  \begin{pmatrix}
    \int_0^{\infty} x \dd\mu_{11}(x) & \cdots & \int_0^{\infty} x \dd\mu_{1N}(x) \\ 
    \vdots & \ddots & \vdots \\ 
    \int_0^{\infty} x \dd\mu_{N1}(x) & \cdots & \int_0^{\infty} x \dd\mu_{NN}(x)
  \end{pmatrix}.
\end{equation*}
The following result is \cite[Theorem 4.3]{MR1367690}.

\begin{theorem} \label{LWC}
  Suppose that $M$ is a matrix valued Radon measure as in \eqref{matrix valued measure} such that the full measure matrix $F_M(\infty)$ is irreducible and has spectral radius $1$. If $L \colon {\R} \to {\R}^N$ is directly Riemann integrable such that $L(x)=(0,\ldots,0)$ for all $x<0$, then the inhomogeneous convolution equation
  \begin{equation*}
    f = f*M + L,
  \end{equation*}
has a unique continuous solution $f=L*U$ vanishing on $(-\infty,0)$. Furthermore, the solution satisfies the following two conditions:
  \begin{enumerate}
    \item If $G_M={\R}$, then
    \begin{equation*}
      \lim_{x \to \infty} f(x)=\biggl(\int_0^{\infty} L_1(t) \dd t, \ldots, \int_0^{\infty} L_N(t) \dd t\biggr) A.
    \end{equation*}
    \item If $G_M=\langle \{\tau\} \rangle$ for some $\tau>0$, then for each $x>0$ it holds that
    \begin{equation*}
      \lim_{n \to \infty} f(x+n\tau) = \biggl(\sum_{k \in {\Z}} L_1(x+k\tau), \ldots, \sum_{k \in {\Z}} L_N(x+k\tau)\biggr)A,
    \end{equation*}
    where $A = (v^\top E u)^{-1}uv^\top$ and $u,v$ are the unique normalized right and left $1$-eigenvectors of $F_M(\infty)$ respectively.
  \end{enumerate}
\end{theorem}

We remark that, in the case $G_M=\langle \{\tau\} \rangle$, the result \cite[Theorem 4.3]{MR1367690} states that $\lim_{n \to \infty} (L*U)_j(x+n\tau) = \sum_{k \in {\Z}} L_j(x-a_{1j}+k\tau)$ for all $a_{ij} \in \spt(\mu_{\gamma(1,j)})$, where $\gamma(1,j)$ is any path from $1$ to $j$ such that $\mu_{\gamma(1,j)} \ne 0$. By recalling \eqref{support set}, we see that $a_{1j}$ is an integer multiple of $\tau$, and the result improves immediately.

\section{Strongly connected inhomogeneous graph-directed self-similar sets} \label{sec:SC-inhomog}

In \S \ref{sec:small-cond}, we prove the first main theorem, Theorem \ref{thm:GIFS-renewal-result}. The method follows a standard route: we reduce the problem to the analysis of a renewal equation, which in our setting is vector valued, and apply the renewal theorem, Theorem \ref{LWC}, of Lau, Wang, and Chu \cite{MR1367690}. The main difficulty is to establish direct Riemann integrability of the function that appears in the renewal equation. Under the SOSC, this amounts to showing that the contribution to the covering number from the overlap between cells is of smaller order than the bulk, so that the difference between the covering number on different scales can be controlled. The second main theorem, Theorem \ref{thm:GIFS-renewal-result2}, is more straightforward; we prove it in \S \ref{sec:large-cond}.

\subsection{Small condensation sets} \label{sec:small-cond}

We work in the setting of Mauldin-Williams graphs introduced in \S\ref{sec:GDIFS}. With a slight abuse of notation, we define
\begin{equation}\label{eq:GMgraphdir}
  G_M=\biggl\langle\bigcup_{i\in V}\bigcup_{n=1}^\infty \{-\log r_\gamma : \gamma\in\Gamma_{i,i}^n\}\biggl\rangle
\end{equation}
for all Mauldin-Williams graphs $G$. The abuse is justified as we will soon see that the above definition agrees with \eqref{support set}.

To prove Theorem \ref{thm:GIFS-renewal-result}, the task is to verify that we are in the setting of the vector-valued renewal theorem explained in \S\ref{sec:vector-renewal} and then apply Theorem \ref{LWC}. We remark that if $G$ were a strongly connected Mauldin-Williams graph, then in the case $G_M=\R$, we could have also applied \cite[Theorem 4]{Lalley89}. We use the vector-valued renewal theorem to extend the result to strongly connected inhomogeneous graphs. We also remark that \cite[Theorem 4]{Lalley89} does not cover the case $G_M=\langle \{\tau\} \rangle$. The role of Lemma \ref{lem:Liestimate} is to ensure the inhomogeneous term $L^*$ in the renewal equation is directly Riemann integrable under the condensation condition \eqref{eq:condensation-finite}; Lemma \ref{estimatelemma} bounds the covering number near the boundary of the open sets so that Lemma \ref{lem:Liestimate} applies. We then deduce Theorem \ref{thm:GIFS-renewal-result} by applying Theorem \ref{LWC}.

We begin by examining the behavior of the covering function. Fix $i\in V=\{1, \ldots, N\}$ and define $L^* = (L_1^*,\ldots,L_N^*) \colon \R \to \R^N$ by setting
\begin{equation}\label{eq:Li}
  L_i^*(t) = \sum_{e\in E_i} N_{e^{-t}}(S_e(K_{t(e)}^C)) - N_{e^{-t}}(K_i^C) 
\end{equation}
for all $t \in \R$. By \eqref{invariantlist}, we see that
\begin{equation*}
N_{e^{-t}}(K_i^C) = \sum_{j=1}^N\sum_{e\in E_{ij}}N_{e^{-t}}(S_e(K_j^C)) - L_i^*(t) = \sum_{j=1}^N\sum_{e\in E_{ij}}N_{e^{-t}r_e^{-1}}(K_j^C) - L_i^*(t)
\end{equation*}
for all $t\in{\R}$. Let $f^* = (f^*_1,\ldots,f^*_N) \colon \R \to \R^N$ be such that
\begin{equation} \label{eq:def-fstar}
  f^*_j(t) = N_{e^{-t}}(K_j^C)e^{-s_0t}
\end{equation}
for all $t\in{\R}$ and $j\in\{1, \ldots, N\}$. Observe that
\begin{align*}
f^*_j(t-\log r_e^{-1})r_e^{s_0}&=N_{e^{-(t-\log r_e^{-1})}}(K_j^C)e^{-s_0(t-\log r_e^{-1})}e^{\log r_e^{s_0}}\\
&=N_{e^{-(t-\log r_e^{-1})}}(K_j^C)e^{-s_0t}\\
&=N_{e^{-t} r_e^{-1}}(K_j^C)e^{-s_0t}
\end{align*}
for all $e\in E_{ij}$ and hence,
\begin{align*}
\sum_{e\in E_{ij}} N_{e^{-t} r_e^{-1}}(K_j^C)e^{-s_0t}=\sum_{e\in E_{ij}} f^*_j(t-\log r_e^{-1})r_e^{s_0}.
\end{align*}
Furthermore,
\begin{equation} \label{equality1}
\begin{split}
f^*_i(t)&=N_{e^{-t}}(K_i^C)e^{-s_0t}=\sum_{j=1}^N\sum_{e\in E_{ij}} N_{e^{-t}r_e^{-1}}(K_j^C)e^{-s_0t}-L_i^*(t)e^{-s_0t} \\
&=\sum_{j=1}^N\sum_{e\in E_{ij}} f^*_j(t-\log r_e^{-1})r_e^{s_0}-L_i^*(t)e^{-s_0t} \\
&=\sum_{j=1}^N f^*_j*\biggl(\sum_{e\in E_{ij}} r_e^{s_0}\delta_{\log r_e^{-1}}\biggr)(t)-L_i^*(t)e^{-s_0t}
\end{split}
\end{equation}
for all $t\in{\R}$, where $\delta_x$ is the Dirac measure at $x \in \R$. Write $\mu_{ij}=\sum_{e \in E_{ij}} r_e^{s_0} \delta_{\log r_e^{-1}}$ for all $i,j \in \{1,\ldots,N\}$ and let
\begin{equation}\label{eq:defM}
  M =
  \begin{pmatrix}
    \mu_{11} & \cdots & \mu_{N1} \\ \vdots & \ddots  &\vdots \\  \mu_{{1N}} & \cdots & \mu_{NN}
  \end{pmatrix}
\end{equation}
be the corresponding matrix valued Radon measure. Then, by \eqref{equality1},
\begin{equation}\label{eq:f*}
  f^*(t) = (f^**M)(t)-L^*(t)e^{-s_0t}
\end{equation}
for all $t \in \R$. Observe that $F_M(\infty) = (A^{s_0}_G)^\top$ and hence, $\rho(F_M(\infty)) = \rho(A^{s_0}_G) = 1$. Recall also that $F_M(\infty)$ is irreducible if and only if $G$ is strongly connected. Observe also that, with the above choices, the closed subgroup $G_M$ defined in \eqref{eq:GMgraphdir} is the same as the closed subgroup defined in \eqref{support set}.

Before going into the proof of Theorem \ref{thm:GIFS-renewal-result}, we show the following estimate for $L^*$.

\begin{lemma}\label{lem:Liestimate}
  Let $(G,C)$ be a strongly connected inhomogeneous Mauldin-Williams graph satisfying the SOSC with condensation sets $(C_i)_{i \in \{1,\ldots,N\}}$ such that $\int_0^\infty e^{-s_0t}N_{e^{-t}}(C_i)\dd t<\infty$ for all $i \in \{1,\ldots,N\}$, where $s_0 = s_0(G)$ is as in \eqref{eq:s0-def}. Then 
  $$
  \sum_{k=0}^\infty\sup_{t\in[k,k+1]}e^{-s_0t}|L_i^*(t)|<\infty
  $$
  for all $i \in \{1,\ldots,N\}$, where $L_i^*$ is as in \eqref{eq:Li}.
\end{lemma}

The proof of Lemma \ref{lem:Liestimate} relies on the covering number estimate for the neighborhood of the boundary. The estimate is used to control the boundary contribution to $L^*$. Recall that if $A,B \subset \R^d$ and $\delta > 0$, then the $\delta$-neighborhood of $A$ is $[A]_\delta = \{x \in \R^d : |x-y| \le \delta \text{ for some } y \in A\}$, the distance between $A$ and $B$ is $\dist(A,B) = \inf\{|x-y| : x \in A \text{ and } y \in B\}$, and the distance between $x \in \R^d$ and $A$ is $\dist(x,A)=\dist(\{x\},A)$.

\begin{lemma} \label{estimatelemma}
  Let $(G,C)$ be a strongly connected inhomogeneous Mauldin-Williams graph satisfying the SOSC with open sets $(U_i)_{i \in \{1,\ldots,N\}}$ and condensation sets $(C_i)_{i \in \{1,\ldots,N\}}$ such that $\int_0^\infty e^{-s_0t}N_{e^{-t}}(C_i)\dd t<\infty$ for all $i \in \{1,\ldots,N\}$, where $s_0 = s_0(G)$ is as in \eqref{eq:s0-def}. Then
  $$
  \int_0^\infty e^{-s_0t}N_{e^{-t}}(K_i^C\cap[\partial U_i]_{e^{-t}})\dd t<\infty
  $$
  for all $i\in\{1, \ldots, N\}$.
\end{lemma}

\begin{proof}
Since $(G, C)$ satisfies the SOSC, for every $i\in\{1,\ldots,N\}$ there exists a point $x\in K_i\cap U_i$, which has a positive distance from the non-empty compact set $\partial U_i$. Hence, for every $i\in\{1,\ldots,N\}$ there exist $n_0 \in \N$ and $\gamma_i\in\Gamma_i^{n_0}$ such that $x\in S_{\gamma_i}(K_{t(\gamma_i)})\subset S_{\gamma_i}(\overline{U_{t(\gamma_i)}})$ and $S_{\gamma_i}(\overline{U_{t(\gamma_i)}})\cap \partial U_i=\emptyset$. Write
  \begin{equation*}
    R=\min_{i\in\{1,\ldots,N\}}\dist(\partial U_i,S_{\gamma_i}(\overline{U_{t(\gamma_i)}}))>0.
  \end{equation*}
  Without loss of generality, we may assume that $n_0$ is common for every $i$ and $\max_{\gamma\in E^{n_0}}r_{\gamma}<R/\max_{i\in\{1,\ldots,N\}}\mathrm{diam}(U_i)$. 
	
	Let us define now $B^q=(\sum_{\Gamma_{i,k}^{n_0}\setminus\{\gamma_i\}}r_\gamma^q)_{i,k\in\{1,\ldots,N\}}$, where $q$ is the unique solution of $\rho(B^q)=1$. Since both $A_G^{q}$ and $B^q$ are both irreducible, the Perron-Frobenius theorem implies
	\begin{equation} \label{red}
		\rho(B^{s_0})<\rho((A_G^{s_0})^{n_0})=\rho(A_G^{s_0})=1,
	\end{equation}
	and so $q<s_0$. Moreover,  let $u,v\in\R^N$ with strictly positive entries be such that $B^{q}u=u$, $v^\top B^{q}=v^\top$, and $v^\top u=1$. Now, let
  \begin{align*}
    \mathcal{M}_{e^{-t}}(i)=\{\gamma\in\bigcup_{m=1}^\infty\Gamma_i^{m n_0} : &\;r_\gamma R\leq {e^{-t}} < r_{\gamma^-}R \text{ and} \\ &\gamma \text{ does not contain any of $\{\gamma_1,\ldots,\gamma_N\}$}\}
  \end{align*}
and
\begin{align*}
	\mathcal{N}_{e^{-t}}(i)=\{\gamma\in\bigcup_{m=1}^\infty\Gamma_i^{m n_0} :&\;r_\gamma R > {e^{-t}} \text{ and} \\ &\gamma \text{ does not contain any of } \{\gamma_1,\ldots,\gamma_N\} \}
\end{align*}
	for all $t > 0$. Hence, by the definition of the vectors $u$ and $v$ and the choice of $q$,
	\begin{equation}\label{eq:1}
  \begin{split}
  	\frac{\#\mathcal{M}_{e^{-t}}(i) e^{-tq}}{R^q}&\leq(\min_iv_iu_i\cdot\min_{\gamma\in E^{n_0}}r_\gamma^q)^{-1}\sum_{\gamma\in\mathcal{M}_{e^{-t}}(i)}v_{i(\gamma)}r_\gamma^q u_{t(\gamma)}\\ &=(\min_iv_iu_i\cdot\min_{\gamma\in E^{n_0}}r_\gamma^q)^{-1}.
  \end{split}
	\end{equation}
	Write $C_i' = \bigcup_{n=0}^{n_0-1} \bigcup_{\gamma \in \Gamma_i^n} S_{\gamma}(C_{t(\gamma)})$ and
  \begin{align*}
    C_i({e^{-t}}) = \bigcup_{\gamma\in \mathcal{N}_{e^{-t}}(i)} S_{\gamma}(C_{t(\gamma)}')
  \end{align*}
  for all $t > 0$. Observe that, by \eqref{eq:inhomogeneous-def2},
  \begin{equation*}
    K_i^C \subset \bigcup_{m=0}^\infty \bigcup_{\gamma \in \Gamma_i^{mn_0}} \{S_\gamma(C_{t(\gamma)}') : r_\gamma R > {e^{-t}}\} \cup \{S_\gamma(\overline{U_{t(\gamma)}}) : r_\gamma R \le {e^{-t}} < r_{\gamma^-}R \}.
  \end{equation*}
  We claim that
  \begin{equation} \label{eq:goal-inclusion}
    K_i^C\cap[\partial U_i]_{e^{-t}} \subset C_i({e^{-t}}) \cup \bigcup_{\gamma\in\mathcal{M}_{e^{-t}}(i)}S_{\gamma}(\overline{U_{t(\gamma)}}).
  \end{equation}
  Indeed, if this was not the case, then there exist $x\in K_i^C\cap[\partial U_i]_{e^{-t}}$ and $\gamma\in\bigcup_{m=1}^\infty\Gamma_i^{mn_0}$ such that $\gamma=\gamma'\gamma_i\gamma''$ for some $i\in \{1,\ldots,N\}$ and $x \in S_\gamma(C_{t(\gamma)}')$ with $r_\gamma R > {e^{-t}}$ or $x\in S_\gamma(\overline{U_{t(\gamma)}})$ with $r_\gamma R\leq {e^{-t}}<r_{\gamma^-}R$. Then, by the SOSC,
	\[
	\begin{split}
	{e^{-t}}&\geq \dist(x,\partial U_i)\geq \dist(S_{\gamma}(\overline{U_{t(\gamma)}}),\partial U_i)\geq \dist(S_{\gamma}(\overline{U_{t(\gamma)}}),S_{\gamma'}(\partial U_{t(\gamma')}))\\
	&=r_{\gamma'}\dist(S_{\gamma_i\gamma''}(\overline{U_{t(\gamma)}}),\partial U_{t(\gamma')})\geq Rr_{\gamma'}>{e^{-t}},
  \end{split}
  \]
which, as $C_i \subset \overline{U_i}$, is a contradiction.

Observe that, by $q<s_0$ and the fact that all matrix norms are equivalent, we have
\begin{equation*}
  D = \sum_{m=0}^\infty \sum_{\gamma \in \Gamma_i^{mn_0}} \{ r_\gamma^{s_0} : \gamma \text{ does not contain any of } \{\gamma_1,\ldots,\gamma_N\} \} < \infty.
\end{equation*}
Hence, 
\begin{equation}\label{eq:first-sum-estimate}
	\begin{split}
  \int_0^\infty e^{-s_0t}&N_{e^{-t}}(C_i({e^{-t}}))\dd t \le \int_0^\infty\sum_{\gamma\in\mathcal{N}_{e^{-t}}(i)} e^{-s_0t}N_{{e^{-t}}/r_\gamma}(C_{t(\gamma)}')\dd t \\ 
  &\le \sum_{\atop{\gamma \in \bigcup_{m=1}^\infty\Gamma_i^{mn_0}}{\text{$\gamma$ does not contain any of $\{\gamma_i\}_{i\in V}$}}}\int_{-\log(r_\gamma R)}^\infty e^{-s_0t}N_{{e^{-t}}/r_\gamma}(C_{t(\gamma)}')\dd t\\
  &=\sum_{\atop{\gamma \in \bigcup_{m=1}^\infty\Gamma_i^{mn_0}}{\text{$\gamma$ does not contain any of $\{\gamma_i\}_{i\in V}$}}}r_{\gamma}^{s_0}\int_{-\log(R)}^\infty e^{-s_0t}N_{e^{-t}}(C_{t(\gamma)}')\dd t\\
  &\le D\cdot \#\biggl(\bigcup_{n=0}^{n_0-1}\Gamma_i^n\biggr)\cdot\max_{i\in V}\int_0^\infty e^{-s_0t}N_{e^{-t}}(C_i)\dd t<\infty
\end{split}
\end{equation}
Finally, by \eqref{eq:goal-inclusion}, \eqref{eq:first-sum-estimate}, and the choice of $n_0$ and \eqref{eq:1}, it is easy to see that 
\begin{align*}
  \int_0^\infty e^{-s_0t}&N_{e^{-t}}(K_i^C\cap[\partial U_i]_{e^{-t}})\dd t \leq \int_0^\infty e^{-s_0t}(N_{e^{-t}}(C_i(e^{-t}))+ \#\mathcal{M}_{e^{-t}}(i)\cdot \#E^{n_0})\dd t \\ 
  &\leq \int_0^\infty e^{-s_0t}N_{e^{-t}}(C_i(e^{-t}))\dd t + \frac{R^{q}\cdot \#E^{n_0}}{\min_iv_iu_i\cdot\min_{\gamma\in E^{n_0}}r_\gamma^q}\int_0^\infty e^{(q-s_0)t}\dd t,
\end{align*}
from which the claim follows.
\end{proof}

With Lemma \ref{estimatelemma} we can now prove Lemma \ref{lem:Liestimate}.

\begin{proof}[Proof of Lemma~\ref{lem:Liestimate}]
Let $(U_i)_{i \in \{1,\ldots,N\}}$ be the list of non-empty bounded open sets given by the SOSC. By the definition \eqref{eq:Li},
  \[
  	L_i^*(t)=\sum_{e\in E_i} N_{e^{-t}}(S_e(K_{t(e)}^C)) - N_{e^{-t}}(K_i^C)
\]
for all $t \in \R$. Observe that 
\[
\begin{split}
  N_{e^{-t}}(S_e(K_{t(e)}^C))&\leq N_{e^{-t}}(S_e(K_{t(e)}^C)\setminus[S_e(\partial U_{t(e)})]_{e^{-t}/2})\\ &\qquad\quad+N_{e^{-t}/2}(S_e(K_{t(e)}^C)\cap[S_e(\partial U_{t(e)})]_{e^{-t}/2}).
\end{split}
\]
Indeed, one can cover the set $S_e(K_{t(e)}^C)$ by first covering $S_e(K_{t(e)}^C)\setminus[S_e(\partial U_{t(e)})]_{e^{-t}/2}$ with balls of radius $e^{-t}$ and then doubling the radius of balls in the cover of $S_e(K_{t(e)}^C)\cap[S_e(\partial U_{t(e)})]_{e^{-t}/2}$, where balls have radius $e^{-t}/2$.

On the other hand, since 
$$
\dist(S_e(K_{t(e)}^C)\setminus[S_e(\partial U_{t(e)})]_{e^{-t}/2},S_{e'}(K_{t(e')}^C)\setminus[S_{e'}(\partial U_{t(e')})]_{e^{-t}/2})>e^{-t}
$$
whenever $e,e'\in E_i$ with $e\neq e'$, one can see that
\[
\begin{split}
N_{e^{-t}}(K_i^C)&\geq N_{e^{-t}}\biggl(\bigcup_{e\in E_i}S_e(K_{t(e)}^C)\setminus[S_e(\partial U_{t(e)})]_{e^{-t}/2}\biggr)\\
&=\sum_{e\in E_i}N_{e^{-t}}(S_e(K_{t(e)}^C)\setminus[S_e(\partial U_{t(e)})]_{e^{-t}/2}).
\end{split}
\]
Therefore,
\begin{equation}\label{eq:finalestimate1}
\begin{split}
  L_i^*(t)&\leq \sum_{e\in E_i}N_{e^{-t}/2}(S_e(K_{t(e)}^C)\cap[S_e(\partial U_{t(e)})]_{e^{-t}/2})\\
  &=\sum_{e\in E_i}N_{e^{-t}/2}(S_e(K_{t(e)}^C)\cap S_e([\partial U_{t(e)}]_{r_e^{-1}e^{-t}/2}))\\
  &=\sum_{e\in E_i}N_{r_e^{-1}e^{-t}/2}(K_{t(e)}^C\cap[\partial U_{t(e)}]_{r_e^{-1}e^{-t}/2})\\
\end{split}
\end{equation}
for all $t\geq0$. On the other hand,
\begin{equation}\label{eq:finalestimate2}
\begin{split}
  L_i^*(t)&=\sum_{e\in E_i} N_{e^{-t}}(S_e(K_{t(e)}^C)) - N_{e^{-t}}(K_i^C)\\
  &=\sum_{e\in E_i} N_{e^{-t}}(S_e(K_{t(e)}^C)) - N_{e^{-t}}(\bigcup_{e\in E_i}S_e(K_{t(e)}^C)\cup C_i)\\
  &\geq\sum_{e\in E_i} N_{e^{-t}}(S_e(K_{t(e)}^C)) - \biggl(\sum_{e\in E_i}N_{e^{-t}}(S_e(K_{t(e)}^C))+N_{e^{-t}}(C_i)\biggr)\\
  &=-N_{e^{-t}}(C_i).
\end{split}
\end{equation}
Combining of \eqref{eq:finalestimate1} and \eqref{eq:finalestimate2}, it is enough to show for the maps $t\mapsto e^{-s_0t}N_{e^{-t}}(C_i)$ and $t\mapsto e^{-s_0t}N_{e^{-t}}(K_i^C\cap[\partial U_i]_{e^{-t}})$ that
$$
  \sum_{k=0}^\infty\sup_{t\in[k,k+1]}e^{-s_0t}N_{e^{-t}}(C_i)<\infty\quad\text{and}\quad\sum_{k=0}^\infty\sup_{t\in[k,k+1]}e^{-s_0t}N_{e^{-t}}(K_i^C\cap[\partial U_i]_{e^{-t}})<\infty.
$$
This follows from Lemma~\ref{estimatelemma} and the assumption on the condensation sets together with the fact that there exists a constant $C>0$ such that
$$
  N_{e^{-(t+\tau)}}(A)\leq CN_{e^{-t}}(A)
$$
for every $t>0$, $\tau\in[0,1]$, and every bounded $A\subset\R^d$.
\end{proof}

Having Lemma \ref{lem:Liestimate} and Theorem \ref{LWC} at our disposal, we are now ready to prove Theorem \ref{thm:GIFS-renewal-result}.

\begin{proof}[Proof of Theorem~\ref{thm:GIFS-renewal-result}] Recall the definitions of $f^*_i$ and $L^*_i$ from \eqref{eq:def-fstar} and \eqref{eq:Li}. By \eqref{eq:f*}, it is tempting to try to apply Theorem~\ref{LWC} with the functions $t \mapsto f_i^*(t)$ and $t \mapsto -L_i^*(t)e^{-s_0t}$. Unfortunately, this does not work since the functions do not vanish for $t<0$. Therefore, let us define $f=(f_1,\ldots,f_N) \colon \R \to \R^N$ and $L=(L_1,\ldots,L_N) \colon \R \to \R^N$ by setting
\begin{equation*}
f_i(t)=
\begin{cases}
f^*_i(t), &\text{if } t\geq 0, \\
0, &\text{if } t<0,
\end{cases}
\end{equation*}
and
\begin{equation*}
L_i(t)=
\begin{cases}
-L_i^*(t)e^{-ts_0}+\sum_{j=1}^N\sum_{e\in E_{ij}\,:\,t<\log r_e^{-1}} f^*_j(t-\log {r_e}^{-1})r_e^{s_0}, &\text{if } t\geq 0, \\
0, &\text{if } t<0,
\end{cases}
\end{equation*}
for all $i\in\{1,\ldots, N\}$. Let us first show that $f$ and $L$ satisfy the inhomogeneous convolution equation
\begin{equation}\label{eq:wewant}
f=f*M+L,
\end{equation}
where $M$ is the matrix valued Radon measure defined in \eqref{eq:defM}. If $t<0$, then \eqref{eq:wewant} holds trivially. We may thus assume that $t\geq0$. If $i\in\{1,\ldots,N\}$, then
\[
\begin{split}
	f_i(t)&=f_i^*(t)=\sum_{j=1}^N\sum_{e\in E_{ij}} f_j^*(t-\log r_e^{-1})r_e^{s_0}-L_i^*(t)e^{-ts_0}\\
	&=\sum_{j=1}^N\sum_{e\in E_{ij}} f_j(t-\log r_e^{-1})r_e^{s_0}+\sum_{j=1}^N\sum_{e\in E_{ij}\,:\,t<\log r_e^{-1}} f^*_j(t-\log {r_e}^{-1})r_e^{s_0}-L_i^*(t)e^{-ts_0}\\
	&=\sum_{j=1}^N(f_j*\mu_{i,j})(t)+L_i(t),
\end{split}\]
where in the second equality we applied \eqref{equality1}.

Let us then show that $L$ is directly Riemann-integrable. Since each $L_i$ is bounded, and continuous outside of a countable set, $L_i$ is Riemann integrable on every compact interval. By the triangle inequality,
\begin{equation} \label{estioferrortwo}
  |L_i(t)| \leq |L_i^*(t)|e^{-ts_0}+\sum_{j=1}^N\sum_{e\in E_{ij}\,:\,t<\log r_e^{-1}} f^*_j(t-\log r_e^{-1})r_e^{s_0}
\end{equation}
for all $t\geq0$. If $t\geq \max_{e\in E_i} \log r_e^{-1}$, then the sum in the right-hand side of \eqref{estioferrortwo} is zero and if $0\leq t\leq \max_{e\in E_i} \log r_e^{-1}$, then the sum can have only a finite number of values. Hence, the direct Riemann integrability of $L_i$ follows by Lemma~\ref{lem:Liestimate}.

Since $L$ is directly Riemann integrable and $f$ satisfies \eqref{eq:wewant}, we may apply Theorem~\ref{LWC} to get information on the limiting behavior of $f^*_i$ defined in \eqref{eq:def-fstar}. Indeed, if $G_M=\R$, then there exists a constant $h_i$ such that
\begin{equation*}
  \lim_{t\to\infty}f^*_i(t)=\lim_{t\to\infty}f_i(t)=h_i,
\end{equation*}
and if $G_M=\langle\{\tau\}\rangle$ for some $\tau>0$, then there exists a $\tau$-periodic function $h_i\colon\R\to\R$ such that
\begin{equation*}
  \lim_{n\to\infty}f^*_i(x+n\tau)=\lim_{n\to\infty}f_i(x+n\tau)=h_i(x)
\end{equation*}
for every $x\in[0,\tau)$. It remains to show that $h_i>0$ and $h_i \colon \R \to [\delta,\infty)$, respectively. But this follows since, by Theorem~\ref{thm:MW-OSC-SOSC} and \eqref{eq:inhomogeneous-def2},
\begin{equation*}
  0<\mathcal{H}^{s_0}(K_i)\leq \mathcal{H}^{s_0}(K_i^C) \le \liminf_{t\to\infty}N_{e^{-t}}(K_i^C)e^{-ts_0}=\liminf_{t\to\infty}f_i(t).
\end{equation*}
Hence, if $G_M=\R$, then the constant $h_i$ is positive, and if $G_M=\langle\{\tau\}\rangle$, then we may choose $\delta=\min_{i \in \{1,\ldots,N\}}\inf_{x\in[0,\tau]}h_i(x)>0$.

Finally, we show the claims for $K^C$. Observe that
\begin{equation}\label{eq:needforcor}
  \lim_{t\to\infty}e^{-s_0t}(N_{e^{-t}}(K_i^C)-N_{e^{-t}}(K_i^C\setminus[\partial U_i]_{e^{-t}}))=0
\end{equation}
for all $i \in \{1,\ldots,N\}$. Indeed, since
\begin{equation*}
  N_{e^{-t}}(K_i^C\setminus[\partial U_i]_{e^{-t}})\leq N_{e^{-t}}(K_i^C)\leq N_{e^{-t}}(K_i^C\setminus[\partial U_i]_{e^{-t}})+N_{e^{-t}}(K_i^C\cap[\partial U_i]_{e^{-t}}),
\end{equation*}
the equation \eqref{eq:needforcor} follows by applying Lemma~\ref{estimatelemma}. Let $(t_n)_{n \in \N}$ be a sequence of positive real numbers converging to $\infty$ such that $\lim_{n\to \infty} e^{-s_0t_n}N_{e^{-t_n}}(K_{i}^C)$ exists for all $i\in\{1,\ldots,N\}$. It follows that
\begin{equation*}
  \limsup_{n\to \infty} e^{-s_0t_n}N_{e^{-t}}(K^C) \le \limsup_{n\to \infty} \sum_{i=1}^N e^{-s_0t_n}N_{e^{-t_n}}(K_{i}^C).
\end{equation*}
By \eqref{eq:needforcor} and the fact that $\mathrm{dist}(K_i^C\setminus[\partial U_i]_{e^{-t}},K_j^C\setminus[\partial U_j]_{e^{-t}})>e^{-t}$ whenever $i\neq j$, we also have
\begin{equation*}
\begin{split}
  \liminf_{n\to \infty} e^{-s_0t_n}N_{e^{-t_n}}(K^C) &\ge \liminf_{n\to \infty} e^{-s_0t_n}N_{e^{-t_n}}\biggl(\bigcup_{i=1}^NK_{i}^C\setminus[\partial U_i]_{e^{-t_n}}\biggr)\\
  &=\liminf_{n\to \infty}\sum_{i=1}^N e^{-s_0t_n}N_{e^{-t_n}}(K_{i}^C\setminus[\partial U_i]_{e^{-t_n}})\\
  &=\liminf_{n\to \infty} \sum_{i=1}^N e^{-s_0t_n}N_{e^{-t_n}}(K_{i}^C).
\end{split}
\end{equation*}
Therefore,
\begin{equation*}
  \lim_{n\to\infty}e^{-s_0t_n}N_{e^{-t}}(K^C)=\lim_{n\to\infty}\sum_{i=1}^N e^{-s_0t_n}N_{e^{-t_n}}(K_{i}^C).
\end{equation*}
If $G_M=\R$, then this gives the claim, and if $G_M=\langle\{\tau\}\rangle$, the claim the follows by choosing $t_n=n\tau+y$ for all $n \in \N$.
\end{proof}

\subsection{A large condensation set} \label{sec:large-cond}

The remaining part of the paper is devoted to the proof of Theorem~\ref{thm:GIFS-renewal-result2}. Let us assume that there exists $i\in V=\{1,\ldots,N\}$ such that $\int_0^\infty e^{-s_0t}N_{e^{-t}}(C_i)\dd t=\infty$. Fix $\rho<r_{\min}=\min_{e\in E}r_{e}$ and notice that 
\begin{equation}\label{eq:neededto2}
\int_0^\infty e^{-s_0t}N_{e^{-t}}(C_i)\dd t=\infty\quad\Leftrightarrow\quad\sum_{k=1}^\infty \rho^{ks_0}N_{\rho^k}(C_i)=\infty.
\end{equation}
Let $j\in \{1,\ldots,N\}$ and write 
$$
\mathcal{K}_k(j)=\{\gamma\in\Gamma_j^*:r_\gamma\leq \rho^k< r_{\gamma_-}\}
$$
for all $k \in \N$. By the choice of $\rho$, we have $\mathcal{K}_k(j)\cap\mathcal{K}_m(j)=\emptyset$ whenever $k\neq m$. Let $A_G^{s_0}$ be the strongly irreducible matrix defined in \S \ref{sec:GDIFS}. By the Perron-Frobenius theorem, there exists vectors $u,v\in\R^N$ with strictly positive entries such that
$$
\mathds{1}^\top v=1,\quad v^\top u=1,\quad A_G^{s_0}u=u,\quad v^\top A_G^{s_0}=v^\top.
$$
Let us define a Markov measure on $\Gamma$ as follows: for every $n \in \N$ and a finite path $\gamma = (e_1,\ldots,e_n) \in \Gamma^n$ let
$$
\mu([e_1,\ldots,e_n])=v_{i(e_1)}r_{e_1}^{s_0}\cdots r_{e_n}^{s_0}u_{t(e_n)},
$$
where $[e_1,\ldots,e_n] = \{\gamma \in \Gamma : \gamma|_n = (e_1,\ldots,e_n)\}$, and extend it to a measure by Kolmogorov's extension theorem. Since $\mathcal{K}_k(j)$ forms a partition of $\Gamma_j$, we get
$$
\sum_{\gamma\in\mathcal{K}_k(j)}\mu([\gamma])=v_{j},
$$
for every $j\in \{1,\ldots,N\}$ and $k\in\N$. 

By the strong condensation open set condition, there exists $\delta>0$ such that for every $\gamma,\gamma'\in\Gamma^*$ with $\gamma\neq\gamma'$ we have
$$
\dist(S_{\gamma}(C_{t(\gamma)}),S_{\gamma'}(C_{t(\gamma')}))>\delta r_{\gamma\wedge\gamma'},
$$
where we recall that $\gamma\wedge\gamma'$ is the common beginning of the paths $\gamma$ and $\gamma'$. Since $(G,C)$ is strongly connected, there exists $q\geq1$ such that for every $j,\ell\in V$ there exists a path $\alpha(j,\ell)$ of length $q$ with initial vertex $j$ and terminal vertex $\ell$. Fix $t>0$ and let $k\geq0$ be such that $\delta r_{\min}^q\rho^k>e^{-t}\geq \delta r_{\min}^q\rho^{k+1}$. It follows that
$$
\dist(S_{\gamma}\circ S_{\alpha(t(\gamma),i)}(C_i),S_{\gamma'}\circ S_{\alpha(t(\gamma'),i)}(C_i))>e^{-t}
$$
for every $\gamma,\gamma'\in\bigcup_{\ell=0}^k\mathcal{K}_\ell(j)$ with $\gamma\neq\gamma'$. Therefore, by recalling \eqref{eq:inhomogeneous-def2}, we get that
\begin{align*}
	e^{-s_0t}N_{e^{-t}}(K_j^C)&\geq e^{-s_0t}N_{e^{-t}}\biggl(\bigcup_{\gamma\in\Gamma_j^*}S_{\gamma}(C_{t(\gamma)})\biggr)\\
	&\geq\sum_{\ell=0}^k\sum_{\gamma\in\mathcal{K}_\ell(j)}e^{-s_0t}N_{e^{-t}}(S_{\gamma}\circ S_{\alpha(t(\gamma,i))}(C_i))\\
	&=\sum_{\ell=0}^k\sum_{\gamma\in\mathcal{K}_\ell(j)}e^{-s_0t}N_{e^{-t}/(r_{\gamma}r_{\alpha(t(\gamma),i)})}(C_i).
\end{align*}
Since $\delta r_{\min}^{q}\rho^{k-\ell+1}\leq e^{-t}/(r_{\gamma}r_{\alpha(t(\gamma),i)})\leq \delta \rho^{k-\ell}$ for every $\gamma\in \mathcal{K}_{\ell}(j)$, there exists a constant $c>0$ such that
\begin{align*}
	\sum_{\ell=0}^k\sum_{\gamma\in\mathcal{K}_\ell(j)}e^{-s_0t}&N_{e^{-t}/(r_{\gamma}r_{\alpha(t(\gamma),i)})}(C_i)\geq c\sum_{\ell=0}^k\sum_{\gamma\in\mathcal{K}_\ell(j)}e^{-s_0t}N_{\rho^{k-\ell}}(C_i)\\
	&= c\sum_{\ell=0}^k\sum_{\gamma\in\mathcal{K}_\ell(j)}r_{\gamma}^{s_0}\Bigl(\frac{e^{-t}}{r_{\gamma}}\Bigr)^{s_0}N_{\rho^{k-\ell}}(C_i)\\
	&\geq c\delta^{s_0}r_{\min}^{qs_0}\sum_{\ell=0}^k\sum_{\gamma\in\mathcal{K}_\ell(j)}r_{\gamma}^{s_0}\rho^{(k-\ell)s_0}N_{\rho^{k-\ell}}(C_i)\\
	&\geq c\delta^{s_0}r_{\min}^{qs_0}(\max_{i}v_iu_i)^{-1}\min_i v_i\cdot\sum_{\ell=0}^k\rho^{(k-\ell)s_0}N_{\rho^{k-\ell}}(C_i)\to\infty
\end{align*}
as $k\to\infty$ by \eqref{eq:neededto2}, which completes the proof of Theorem~\ref{thm:GIFS-renewal-result2}.

\bibliographystyle{abbrv}
\bibliography{Bibliography}

\end{document}